\documentclass[11pt]{article}
\usepackage{amsfonts}
\usepackage{latexsym}
\usepackage{cite}
\usepackage{amsmath,amsfonts,latexsym,amssymb,color}
\usepackage[mathscr]{eucal}
\usepackage{cases}
\usepackage{amsthm}

\usepackage[bf,small]{caption2}
\usepackage{float}
\usepackage{graphicx}
\usepackage{amsmath}
\usepackage{amssymb}
\usepackage[all]{xy}

\newtheorem{theorem}{theorem}[section]
\newtheorem{thm}[theorem]{Theorem}
\newtheorem{lem}[theorem]{Lemma}

\newtheorem{rmk}[theorem]{Remark}

\begin{document}

\title{\textbf{Automorphisms of metacyclic groups}}
\author{\Large Haimiao Chen
\footnote{Email: \emph{chenhm@math.pku.edu.cn}}  \\
\normalsize \em{Department of Mathematics, Beijing Technology and Business University, Beijing, China}\\
\Large Yueshan Xiong
\footnote{Email: \emph{xiongyueshan@gmail.com}}\\
\normalsize \em{School of Mathematics and Statics, Huazhong University of Science and Technology, Wuhan, China}\\
\Large Zhongjian Zhu
\footnote{Email: \emph{zhuzhongjianzzj@126.com}}\\
\normalsize \em{School of Mathematics and Information Science, Wenzhou University, Wenzhou, China}}
\date{}
\maketitle

\begin{abstract}
  A metacyclic group $H$ can be presented as $\langle \alpha,\beta\mid \alpha^{n}=1, \ \beta^{m}=\alpha^{t}, \ \beta\alpha\beta^{-1}=\alpha^{r}\rangle$ for some $n,m,t,r$. Each endomorphism $\sigma$ of $H$ is determined by
  $\sigma(\alpha)=\alpha^{x_{1}}\beta^{y_{1}}, \sigma(\beta)=\alpha^{x_{2}}\beta^{y_{2}}$ for some integers $x_{1},x_{2},y_{1},y_{2}$. We give sufficient and necessary conditions on $x_{1},x_{2},y_{1},y_{2}$ for $\sigma$ to be an automorphism.

  \medskip
  \noindent {\bf Keywords:}  automorphism, metacyclic group, linear congruence equation. \\
  {\bf MSC 2010:} 20D45.
\end{abstract}

\section{Introduction}

A finite group $G$ is {\it metacyclic} if it contains a cyclic, normal subgroup $N$ such that $G/N$ is also cyclic. In some sense, metacyclic groups  can be regarded as the simplest ones other than abelian groups.

As a natural object, the automorphism group of a metacyclic group has been widely studied.
In 1970, Davitt \cite{Da70} showed that if $G$ is a metacyclic $p$-group with $p\neq 2$, then the order $G$ divides that of ${\rm Aut}(G)$.
In 2006, Bidwell and Curran \cite{BC06} found the order and the structure of ${\rm Aut}(G)$ when $G$ is a split metacyclic $p$-group with $p\neq 2$, and in 2007, Curran \cite{Cu07} obtained similar results for split metacyclic $2$-groups.
In 2008, Curran \cite{Cu08} determined ${\rm Aut}(G)$ when $G$ is a nonsplit metacyclic $p$-group with $p\neq 2$. In 2009, Golasi\'nski and Gon\c{c}alves \cite{GG09} determined ${\rm Aut}(G)$ for any split metacyclic group $G$.
The case of nonsplit metacyclic 2-groups remains unknown.

In this paper we aim at writing down all of the automorphisms for a general metacyclic group. One of our main motivations stems from the study of regular Cayley maps on metacyclic groups (see \cite{Ch17}), which requires an explicit formula for a general automorphism.

It is well-known (see Section 3.7 of \cite{Za56}) that each metacyclic group can be presented as
\begin{align}
\langle \alpha,\beta\mid \alpha^{n}=1,\ \beta^{m}=\alpha^{t},\ \beta\alpha\beta^{-1}=\alpha^{r}\rangle \label{eq:presentation}
\end{align}
for some positive integers $n,m,r,t$ satisfying
\begin{align}
r^{m}-1\equiv t(r-1)\equiv 0\pmod{n}.   \label{eq:condition}
\end{align}
Denote this group by $H=H(n,m;t,r)$.
There is an extension
$$1\to\mathbb{Z}/n\mathbb{Z}\to H\to\mathbb{Z}/m\mathbb{Z}\to 1,$$
where $\mathbb{Z}/n\mathbb{Z}\cong\langle \alpha\rangle\lhd H$ and $\mathbb{Z}/m\mathbb{Z}\cong H/\langle \alpha\rangle$.
It may happen that two groups given by different values of $n,m,t,r$ are isomorphic. A complete classification (up to isomorphism) for finite metacyclic groups was obtained by Hempel \cite{He00} in 2000.

In the presentation (\ref{eq:presentation}), we may assume $t\mid n$ which we do from now on. To see this, choose $u,v$ such that $un+vt=(n,t)$, then $(v,n/(n,t))=1$. Let $w$ be the product of all prime factors of $m$ that do not divide $v$ and let $v'=v+wn/(n,t)$, then $(v',m)=1$. Replacing $\beta$ by $\check{\beta}=\beta^{v'}$,
we get another presentation:
$H=\langle \alpha,\check{\beta}\mid \alpha^{n}=1, \check{\beta}^{m}=\alpha^{(n,t)}, \check{\beta}\alpha\check{\beta}^{-1}=\alpha^{r^{v'}}\rangle.$

Obviously each element can be written as $\alpha^{u}\beta^{v}$; note that $\alpha^{u}\beta^{v}=1$ if and only if
$m\mid v$ and $n\mid (u+tv/m)$.
Each endomorphism $\sigma$ of $H$ is determined by
$\sigma(\alpha)=\alpha^{x_{1}}\beta^{y_{1}}, \sigma(\beta)=\alpha^{x_{2}}\beta^{y_{2}}$
for some integers $x_{1},x_{2},y_{1},y_{2}$.
The main result of this paper gives sufficient and necessary conditions on $x_{1},x_{2},y_{1},y_{2}$, for $\sigma$ to be an automorphism.
They consist of two parts, ensuring $\sigma$ to be invertible and well-defined, respectively. Skillfully using elementary number theoretic techniques, we manage to reduce the second part to linear congruence equations. It turns out that the situation concerning the prime 2 is quite subtle,
and this reflects the difficulty in determining the automorphism groups of nonsplit metacyclic 2-groups.

\bigskip

\textbf{Notation and Convention}
\begin{itemize}
  \item For an integer $N>0$, denote  $\mathbb{Z}/N\mathbb{Z}$ by $\mathbb{Z}_{N}$ and regard it as a quotient ring of $\mathbb{Z}$.
        For $u\in\mathbb{Z}$, denote its image under the quotient $\mathbb{Z}\twoheadrightarrow\mathbb{Z}_{N}$ also by $u$.
  \item Given integers $u,s$ with $u>0$, set $[u]_{s}=1+s+\cdots+s^{u-1}$, so that $(s-1)[u]_s=s^u-1$;
        for a prime number $p$, let $\deg_{p}(u)$ denote the largest integer $s$ with $p^{s}\mid u$.
  \item Denote $\alpha^{u}$ by $\exp_{\alpha}(u)$ when the expression for $u$ is too long.
  \item To avoid subtleties, we assume $x_1,x_2,y_1,y_2$ to be positive, and usually write an element of $H$ as $\alpha^u\beta^v$ with $u,v>0$.
\end{itemize}

\section{Determining all automorphisms}

\subsection{Preparation}

\begin{lem} \label{lem:simplify}
If $s>1$ with $\deg_{p}(s-1)=\ell\ge 1$ and $x>0$ with $\deg_{p}(x)=u\ge 0$, then
\begin{enumerate}
  \item[\rm(I)]
  $[x]_{s}\equiv
  \begin{cases}  x, &p\ne 2 \ \text{or} \ u=0 , \\  (1+2^{\ell-1})x, &p=2 \ \text{and} \ u>0 \end{cases}  \pmod{p^{\ell+u}}$;
  \item[\rm(II)]
  $s^{x}-1\equiv
  \begin{cases}  (s-1)x, &p\ne 2 \ \text{or} \ u=0 , \\ (s-1+2^{2\ell-1})x, &p=2 \ \text{and} \ u>0  \end{cases}  \pmod{p^{2\ell+u}}$.
 \end{enumerate}
\end{lem}

\begin{proof}
We only prove (I), then (II) follows from the identity $(s-1)[x]_{s}=s^{x}-1$.

If $u=0$, then $s\equiv 1\pmod{p^{\ell+u}}$, so $[x]_s\equiv x\pmod{p^{\ell+u}}$.

Let us assume $u>0$. Write $s=1+p^{\ell}h$ with $p\nmid h$.
Note that
\begin{align*}
\deg_{p}\left({p^{u}\choose j}\right)
&=\deg_{p}\left(\frac{(p^{u})!}{j!(p^{u}-j)!}\right) =\sum\limits_{i=0}^{j-1}\deg_{p}(p^{u}-i)-\sum\limits_{i=1}^{j}\deg_{p}(i)  \\
&=u-\deg_{p}(j)+\sum\limits_{i=1}^{j-1}(\deg_p(p^u-i)-\deg_p(i))  \\
&=u-\deg_p(j).
\end{align*}

If $p\ne 2$, then
\begin{align*}
[p^{u}]_{s}
=\sum\limits_{i=0}^{p^{u}-1}(1+p^{\ell}h)^{i}=\sum\limits_{i=0}^{p^{u}-1}\sum\limits_{j=0}^{i}{i\choose j}(p^{\ell}h)^{j}
=\sum\limits_{j=1}^{p^{u}}{p^{u}\choose j}(p^{\ell}h)^{j-1} \equiv p^{u}\pmod{p^{\ell+u}},
\end{align*}
using that for all $j\ge 2$,
\begin{align*}
\deg_{p}\left({p^{u}\choose j}\right)=u-\deg_p(j) \geq u-(j-2)\ell=(\ell+u)-(j-1)\ell.
\end{align*}
Hence $s^{p^{u}}=(s-1)[p^{u}]_{s}+1\equiv 1\pmod{p^{\ell+u}}$. Writing $x=p^{u}x'$ with $p\nmid x'$, we have
$$[x]_{s}=[p^{u}]_{s}\cdot\sum\limits_{j=0}^{x'-1}(s^{p^{u}})^{j}\equiv x'[p^{u}]_{s}
\equiv x\pmod{p^{\ell+u}}.$$

If $p=2$, then using that for all $j\ge 3$,
\begin{align*}
\deg_{2}\left({2^{u}\choose j}\right)=u-\deg_2(j) \geq u-(j-2)\ell=(\ell+u)-(j-1)\ell,
\end{align*}
we obtain
\begin{align*}
[2^{u}]_{s}=\sum\limits_{j=1}^{2^{u}}{2^{u}\choose j}(2^{\ell}h)^{j-1}\equiv 2^u+{2^{u}\choose 2}2^{\ell}h \equiv 2^{u}(1+2^{\ell-1}) \pmod{2^{\ell+u}}.
\end{align*}
Hence $s^{2^{u}}=(s-1)[2^{u}]_{s}+1\equiv 1\pmod{2^{\ell+u}}$. Writing $x=2^{u}x'$ with $2\nmid x'$, we have
$$[x]_{s}=[2^{u}]_{s}\cdot\sum\limits_{j=0}^{x'-1}(s^{2^{u}})^{j}\equiv x'[2^{u}]_{s}
\equiv (1+2^{\ell-1})x \pmod{2^{\ell+u}}.$$

\end{proof}

\subsection{The method}

It follows from (\ref{eq:presentation}) that, for $k,u,v,u',v'>0$,
\begin{align}
\beta^{v}\alpha^{u}&=\alpha^{ur^{v}}\beta^{v}, \label{eq:identity1} \\
(\alpha^{u}\beta^{v})(\alpha^{u'}\beta^{v'})&=\alpha^{u+u'r^{v}}\beta^{v+v'},\label{eq:identity2} \\
(\alpha^{u}\beta^{v})^{k}&=\alpha^{u[k]_{r^v}}\beta^{vk}, \label{eq:identity3} \\
[\alpha^{u}\beta^{v},\alpha^{u'}\beta^{v'}]&=\exp_{\alpha}(u'(r^{v}-1)-u(r^{v'}-1)), \label{eq:identity4}
\end{align}
where the notation $[\theta,\eta]=\theta\eta\theta^{-1}\eta^{-1}$ for commutator is adopted.

In view of (\ref{eq:identity4}), the commutator subgroup $[H,H]$ is generated by $\alpha^{r-1}$.
The abelianization $H^{{\rm ab}}:=H/[H,H]$ has a presentation
\begin{align}
\langle \overline{\alpha},\overline{\beta}\mid q\overline{\alpha}=0,\ m\overline{\beta}=t\overline{\alpha}\rangle, \qquad \text{with} \qquad q=(r-1,n),
\end{align}
where additive notations are used and $\overline{\alpha}+\overline{\beta}=\overline{\beta}+\overline{\alpha}$ is implicitly assumed.

\begin{lem}  \label{lem:well-defined}
%Suppose $0\leq x_{1},x_{2}<n$ and $0\leq y_{1},y_{2}<m$.
There exists a homomorphism $\sigma:H\to H$ with
$\sigma(\alpha)=\alpha^{x_{1}}\beta^{y_{1}}$, $\sigma(\beta)=\alpha^{x_{2}}\beta^{y_{2}}$ if and only if
\begin{align}
(r-1,t)y_{1}&\equiv 0\pmod{m}, \label{eq:condition1}  \\
x_{2}[m]_{r^{y_{2}}}+ty_{2}-x_{1}[t]_{r^{y_{1}}}-\frac{ty_{1}}{m}t&\equiv 0\pmod{n}, \label{eq:condition2}\\
x_2(r^{y_{1}}-1)+x_1([r]_{r^{y_{1}}}-r^{y_{2}})+\frac{(r-1)y_{1}}{m}t&\equiv 0\pmod{n}.    \label{eq:condition3}
\end{align}
\end{lem}

\begin{proof}
Sufficient and necessary conditions for $\sigma$ to be well-defined are:
\begin{align*}
\alpha^{x_{1}[n]_{r^{y_{1}}}}\beta^{y_{1}n}&=\sigma(\alpha)^{n}=1, \\
\alpha^{x_{2}[m]_{r^{y_{2}}}}\beta^{y_{2}m}=\sigma(\beta)^{m}&=\sigma(\alpha)^{t}=\alpha^{x_{1}[t]_{r^{y_{1}}}}\beta^{y_{1}t}, \\
\alpha^{x_{2}}\beta^{y_{2}}\alpha^{x_{1}}\beta^{y_{1}}\beta^{-y_{2}}\alpha^{-x_{2}}
=\sigma(\beta)\sigma(\alpha)\sigma(\beta)^{-1}&=\sigma(\alpha)^{r}=\alpha^{x_{1}[r]_{r^{y_{1}}}}\beta^{y_{1}r};
\end{align*}
equivalently,
\begin{align}
ny_{1}&\equiv 0\pmod{m},     & x_{1}[n]_{r^{y_{1}}}+\frac{ny_{1}}{m}t\equiv 0&\pmod{n},  \label{eq:1st} \\
ty_{1}&\equiv 0\pmod{m},     & x_{2}[m]_{r^{y_{2}}}+y_{2}t\equiv x_{1}[t]_{r^{y_{1}}}+\frac{ty_{1}}{m}t&\pmod{n}, \label{eq:2nd} \\
(r-1)y_{1}&\equiv 0\pmod{m}, & x_{2}(1-r^{y_{1}})+x_{1}r^{y_{2}}\equiv x_{1}[r]_{r^{y_{1}}}+\frac{(r-1)y_{1}}{m}t&\pmod{n}.  \label{eq:3rd}
\end{align}
Due to $t\mid n$, the first halves of (\ref{eq:1st}), (\ref{eq:2nd}), (\ref{eq:3rd}) are equivalent to the single condition (\ref{eq:condition1}).
Then the second half of (\ref{eq:1st}) can be omitted: for each prime divisor $p$ of $n$,
if $p\mid r^{y_{1}}-1$, then by Lemma \ref{lem:simplify} (I), $\deg_{p}([n]_{r^{y_{1}}})\geq\deg_{p}(n)$; if $p\nmid r^{y_{1}}-1$, then since $r^{ny_{1}}-1$ is a multiple of $r^{m}-1$, we also have $\deg_{p}([n]_{r^{y_{1}}})=\deg_{p}(r^{ny_{1}}-1)\ge\deg_p(r^m-1)\ge\deg_{p}(n).$
\end{proof}

Let $\Lambda$ denote the set of prime divisors of $nm$, and for each $p\in\Lambda$, denote
\begin{align}
a_{p}=\deg_{p}(n), \qquad
b_{p}=\deg_{p}(m), \qquad
c_{p}=\deg_{p}(t),  \qquad
d_{p}=\deg_{p}(q).  \label{eq:fac}
\end{align}
Subdivide $\Lambda$ as $\Lambda=\Lambda_{1}\sqcup \Lambda_{2}\sqcup\Lambda'$, with
\begin{align}
\Lambda_{1}=\{p\colon d_{p}>0\}, \qquad \Lambda_{2}=\{p\colon a_{p}>0,\ d_{p}=0\}, \qquad  \Lambda'=\{p\colon b_p>0,\ a_{p}=0\}.
\end{align}
Denote
\begin{align}
e=\deg_{2}(r+1).
\end{align}

It follows from $t\mid n$ and $t(r-1)\equiv 0\pmod{n}$ that
\begin{align}
\left\{\begin{array}{ll}
a_{p}-d_{p}\le c_{p}\le a_{p}, &p\in\Lambda_{1}, \\
c_{p}=a_{p}, &p\in\Lambda_{2},
\end{array}\right.   \label{ineq:a}
\end{align}
and it follows from $r^{m}-1\equiv 0\pmod{n}$ and Lemma \ref{lem:simplify} (II) that
\begin{align}
d_{p}+b_{p}\geq a_{p},  \qquad \text{for\ all\ } p\in\Lambda_{1} \ \text{with} \ (p,d_{p})\neq (2,1) \ \text{or\ }(p,d_{p},b_{p})=(2,1,0);
\label{ineq:b1}
\end{align}
finally, when $d_2=1$ and $b_2>0$, Lemma \ref{lem:simplify} (II) applied to
$r^m-1=(r^2)^{m/2}-1$ implies
\begin{align}
e+b_{2}\geq a_{2}.     \label{ineq:b2}
\end{align}
The condition (\ref{eq:condition1}) is equivalent to
\begin{align}
\min\{d_{p},c_{p}\}+\deg_{p}(y_{1})\geq b_{p} \qquad \text{for\ all } \quad  p\in\Lambda.  \label{ineq:m}
\end{align}

Suppose that $x_1,x_2,y_1,y_2$ satisfy the conditions (\ref{eq:condition1}), (\ref{eq:condition2}) and (\ref{eq:condition3}) and let $\sigma$ be the endomorphism of $H$ given in Lemma \ref{lem:well-defined}.
Since $H$ is finite, $\sigma$ is invertible if and only if it is injective, which is equivalent to that both the induced homomorphism $\overline{\sigma}:H^{\rm ab}\to H^{\rm ab}$ and the restriction $\sigma_0:=\sigma|_{[H,H]}$ are injective.

In the remainder of this subsection, let
\begin{align}
w=\frac{ty_1}{m}.
\end{align}

\begin{lem} \label{lem:injective1}
%Let $\delta^{0}_{p}=\min\{\alpha_{p},\delta_{p}\}=\deg_{p}(d)$.
The homomorphism $\overline{\sigma}$ is injective if and only if
\begin{align}
\begin{cases}
p \nmid y_{2}, &p\in\Lambda', \\
p \nmid x_{1}+w,  &p\in\Lambda_{1}\ \text{with}\ b_{p}c_p=0,  \\
p \nmid x_{1}y_{2}-x_{2}y_{1},   &p\in\Lambda_{1} \ \text{with}\  b_{p},c_p>0.
\end{cases}             \label{eq:condition00}
\end{align}
\end{lem}

\begin{proof}
For each $p\in\Lambda'\sqcup\Lambda_{1}$, let
$$H^{{\rm ab}}_{p}=\langle\overline{\alpha}_{p},\overline{\beta}_{p}\rangle, \qquad \text{with} \qquad
\overline{\alpha}_{p}=\frac{tq}{p^{c_p+d_{p}}}\overline{\alpha}, \quad \overline{\beta}_{p}=\frac{mq}{p^{b_{p}+d_p}}\overline{\beta};$$
it is the Sylow $p$-subgroup of $H^{{\rm ab}}$. Then $\overline{\sigma}$ is injective if and only if $\overline{\sigma}_{p}:=\overline{\sigma}|_{H^{{\rm ab}}_{p}}$ is injective for all $p$.
Take an integer $z_p$ with $(t/p^{c_{p}})z_p\equiv 1\pmod{p^{d_{p}}}$.
We have
\begin{align}
\overline{\sigma}_p(\overline{\alpha}_{p})%&=\overline{\sigma}\left(\frac{tq}{p^{c_p+d_{p}}}\overline{\alpha}\right)
&=\frac{tq}{p^{c_p+d_p}}(x_{1}\overline{\alpha}+y_{1}\overline{\beta})=
x_{1}\overline{\alpha}_{p}+\frac{p^{b_p}ty_{1}}{p^{c_{p}}m}\overline{\beta}_{p},  \label{eq:auto-ab1}  \\
\overline{\sigma}_p(\overline{\beta}_{p})%&=\overline{\sigma}\left(\frac{mq}{p^{b_{p}+d_p}}\overline{\beta}\right)
&=\frac{mq}{p^{b_{p}+d_p}}(x_{2}\overline{\alpha}+y_{2}\overline{\beta})=\frac{m}{p^{b_{p}}}z_px_{2}\overline{\alpha}_{p}+y_{2}\overline{\beta}_{p}.
\label{eq:auto-ab2}
\end{align}
%It is easy to see that $\overline{\sigma}$ is injective if and only if $\overline{\sigma}_p$ is injective for each $p$; actually $H^{{\rm ab}}_p$ is the Sylow $p$-subgroup of $H^{{\rm ab}}$.

Let $\check{H}_p=H^{{\rm ab}}_p/pH^{{\rm ab}}_p$, let $\check{\alpha}_p,\check{\beta}_p$ denote the images of $\overline{\alpha}_p,\overline{\beta}_p$ under the quotient homomorphism $H^{{\rm ab}}_p\to\check{H}_p$, and let $\check{\sigma}_p$ denote the endomorphism of $\check{H}_p$ induced from $\overline{\sigma}_p$. Then $\overline{\sigma}_p$ is injective if and only if $\check{\sigma}_p$ is injective. It follows from (\ref{eq:auto-ab1}), (\ref{eq:auto-ab2}) that
\begin{align}
\check{\sigma}_p(\check{\alpha}_{p})&=x_{1}\check{\alpha}_{p}+\frac{p^{b_p}ty_{1}}{p^{c_{p}}m}\check{\beta}_{p}, \label{eq:auto-ab3} \\
\check{\sigma}_p(\check{\beta}_{p})&=\frac{m}{p^{b_{p}}}z_px_{2}\check{\alpha}_{p}+y_{2}\check{\beta}_{p}. \label{eq:auto-ab4}
\end{align}
\begin{itemize}
  \item If $b_p>d_p=0$, then $\check{\alpha}_p=0$, $\check{H}_p=\langle\check{\beta}_p\rangle\cong\mathbb{Z}_p$, and by (\ref{eq:auto-ab4}), $\check{\sigma}_p$ is injective if and only if $p\nmid y_2$.
  \item If $d_p>b_p=0$, then $\check{\beta}_p=p^{c_p}\check{\alpha}_p$, $\check{H}_p=\langle\check{\alpha}_p\rangle\cong\mathbb{Z}_p$, and by (\ref{eq:auto-ab3}), $\check{\sigma}_p$ is injective if and only if $p\nmid x_1+w$.
  \item If $d_p>c_p=0$, then $\check{\alpha}_p=p^{b_p}\check{\beta}_p$, $\check{H}_p=\langle\check{\beta}_p\rangle$, and by (\ref{eq:auto-ab4}), $\check{\sigma}_p$ is injective if and only if $p\nmid mz_px_2+y_2$, which, by (\ref{eq:condition2}), is equivalent to $p\nmid x_1+w$.
  \item If $b_p,c_p,d_p>0$, then $\check{H}_p=\langle\check{\alpha}_p,\check{\beta}_p\rangle\cong\mathbb{Z}_p^2$, and by (\ref{eq:auto-ab3}), (\ref{eq:auto-ab4}), $\overline{\sigma}_p$ is invertible if and only if
      $$0\not\equiv x_1y_2-\frac{p^{b_p}ty_{1}}{p^{c_{p}}m}\cdot\frac{m}{p^{b_{p}}}z_px_{2}\equiv x_1y_2-x_2y_1\pmod{p}.$$
\end{itemize}
\end{proof}

\begin{lem}  \label{lem:injective2}
Suppose $p\nmid x_1y_2-x_2y_1$ for all $p\in\Lambda_1$ with $d_p<a_p$. Then the homomorphism $\sigma_0$ is injective if and only if
\begin{align}
r^{y_1}\equiv 1\pmod{p^{a_p}} \quad \text{and}\quad p\nmid x_1+w \qquad  \text{for\ all} \quad p\in\Lambda_2. \label{eq:condition0}
\end{align}
\end{lem}
\begin{proof}
Note that $\sigma_0(\alpha^{r-1})=\alpha^u,$ with
\begin{align}
u=x_1[r-1]_{r^{y_1}}+(r-1)w.
\end{align}

For each $p\in\Lambda_{1}$ with $d_{p}<a_{p}$, by (\ref{eq:condition3}) we have
\begin{align*}
u&\equiv (1-r^{y_1})x_1[r-1]_{r^{y_1}}+x_1(r^{y_2}-1)-x_2(r^{y_1}-1)\pmod{p^{a_p}} \\
&\equiv (r-1)(x_1y_2-x_2y_1) \pmod{p^{d_p+1}},
\end{align*}
the second line following from $r^{y_j}-1\equiv (r-1)y_j\pmod{p^{2d_p}}$, $j=1,2$. Hence
\begin{align}
\deg_p(u)=d_p.  \label{eq:u}
\end{align}

Thus $\sigma_0$ is injective if and only if $p\nmid u$ for all $p\in\Lambda_2$. For $p\in\Lambda_2$, by (\ref{ineq:a}), (\ref{ineq:m}),
$$\deg_p(w)=c_p+\deg_p(y_1)-b_p\ge c_p=a_p.$$
Hence, if $p\nmid u$ then $p\nmid x_1[r-1]_{r^{y_1}}$ and this implies that $r^{y_1}\equiv 1\pmod{p^{a_p}}$ (by the argument given). On the other hand, if $r^{y_1}\equiv 1\pmod{p}$ then $[r-1]_{r^{y_1}}\equiv r-1\not\equiv 0\pmod{p^{a_p}}$ and hence $p\mid u$ if and only if $p\mid x_1$. Therefore, $\sigma_0$ is injective if and only if $p\nmid u$ if and only if $r^{y_1}\equiv 1\pmod{p^{a_p}}$ and $p\nmid x_1$; the condition $p\nmid x_1$ is equivalent to $p\nmid x_1+w$.
\end{proof}

\begin{rmk}
\rm In order to obtain neat conditions, we prefer $p\nmid x_1+w$ to $p\nmid x_1$.
\end{rmk}

Summarizing, sufficient and necessary conditions for $\sigma$ to be an automorphism are (\ref{eq:condition1}), (\ref{eq:condition2}), (\ref{eq:condition3}), (\ref{eq:condition00}) and (\ref{eq:condition0}).
Let (\ref{eq:condition2})$_{p}$ (resp. (\ref{eq:condition3})$_{p}$) denote the condition (\ref{eq:condition2}) (resp. (\ref{eq:condition3})) with ${\rm mod}\ n$ replaced by ${\rm mod}\ p^{a_{p}}$. Then (\ref{eq:condition2}) (resp. (\ref{eq:condition3})) is equivalent to (\ref{eq:condition2})$_{p}$ (resp. (\ref{eq:condition3})$_{p}$) for all $p\in\Lambda_{1}\sqcup\Lambda_{2}$ simultaneously.

\begin{rmk}  \label{rmk:p-II-I1}
\rm If $p\in\Lambda_{2}$, then $p\neq 2$: otherwise $2\mid n$ but $2\nmid r-1$, contradicting $n\mid r^{m}-1$. Due to (\ref{ineq:a}), (\ref{eq:condition0}), the conditions {\rm(\ref{eq:condition2})}$_{p}$, {\rm(\ref{eq:condition3})}$_{p}$ are equivalent to
$r^{y_{2}-1}\equiv 1\pmod{p^{a_{p}}}$.

If $p\in\Lambda_{1}$ with $d_{p}=a_{p}$, then $r\equiv 1\pmod{p^{a_{p}}}$,
hence (\ref{eq:condition3})$_{p}$ is trivial, and (\ref{eq:condition2})$_{p}$ becomes
$t(x_{1}+w-y_{2})\equiv mx_{2}\pmod{p^{a_{p}}}.$
\end{rmk}

Suppose $p\in\Lambda_1$ with $d_p<a_p$. Note that by (\ref{ineq:b1}), $b_{p}>0$.  We shall  simplify (\ref{eq:condition2})$_p$ and (\ref{eq:condition3})$_p$, with (\ref{eq:condition1}) and (\ref{eq:condition00}) assumed.

By Lemma \ref{lem:simplify} (I), $[r-1]_{r^{y_1}}\equiv r-1\pmod{p^{2d_p}}$ when $p\ne 2$ or $p=2,\deg_2(r^{y_1}-1)>1$. Hence by (\ref{eq:u}),
\begin{align}
p\nmid x_1+w \qquad \text{if} \qquad p\ne 2 \quad \text{or} \quad p=2,\ d_2+\deg_2(y_1)>1.  \label{eq:w}
\end{align}
By (\ref{ineq:a}), (\ref{ineq:b1}), (\ref{ineq:m}),
\begin{align}
\deg_p(y_1)&\ge b_p-d_p\ge a_p-2d_p,  \qquad \text{if} \quad (p,d_p)\ne(2,1), \label{ineq:deg-y1} \\
\deg_p(w)=\deg_p(y_1)+c_p-b_p&\ge c_p-d_p\ge a_p-2d_p,  \qquad \text{if} \quad (p,d_p)\ne(2,1).   \label{ineq:deg-w}
\end{align}
We will use (\ref{eq:w}), (\ref{ineq:deg-y1}), (\ref{ineq:deg-w}) repeatedly.

\begin{lem} \label{lem:deducing}
If $2\neq p\in\Lambda_{1}$, then the conditions {\rm(\ref{eq:condition2})}$_{p}$ and {\rm(\ref{eq:condition3})}$_{p}$ hold if and only if
\begin{align}
mx_{2}&\equiv t(x_{1}+w-y_{2})\pmod{p^{a_{p}}}, \label{eq:lambda1-1} \\
y_{2}&\equiv 1+w\pmod{p^{a_{p}-d_{p}}}.   \label{eq:lambda1-2}
\end{align}
\end{lem}

\begin{proof}
Abbreviate $a_{p},b_{p},c_{p},d_{p},\deg_{p}(x)$ to $a,b,c,d,\deg(x)$, respectively.

Applying Lemma \ref{lem:simplify}, with (\ref{ineq:a}), (\ref{ineq:b1}), (\ref{ineq:deg-y1}) recalled, we obtain
\begin{align*}
r^{y_{1}}\equiv 1+(r-1)y_{1},  \qquad  [t]_{r^{y_{1}}}\equiv t, \qquad [m]_{r^{y_{2}}}\equiv m\pmod{p^a}, \\
[r]_{r^{y_{1}}}=(r^{y_1})^{r-1}+[r-1]_{r^{y_1}}\equiv 1+(r-1)=r\pmod{p^a}.
\end{align*}
Hence (\ref{eq:condition2})$_{p}$ can be simplified as (\ref{eq:lambda1-1}) and (\ref{eq:condition3})$_{p}$ can be rewritten as
\begin{align}
(r-1)y_{1}x_{2}+(r-1)w\equiv (r^{y_{2}}-r)x_{1}\pmod{p^{a}}. \label{eq:condition4}
\end{align}
By (\ref{ineq:deg-y1}) and (\ref{ineq:deg-w}), $\deg((r-1)y_{1}x_{2}+(r-1)w)\ge a-d$, hence
\begin{align}
\deg(y_2-1)+\deg(x_1)=\deg((r^{y_2}-r)x_1)-d\ge a-2d. \label{eq:deg-y2}
\end{align}
By Lemma \ref{lem:simplify} (II), $r^{y_2-1}-1\equiv (r-1)(y_2-1)\pmod{p^{a-\deg(x_1)}}$, and then
%Using $(r-1)^2(y_2-1)x_1\equiv 0\pmod{p^a}$, we obtain
\begin{align*}
(r^{y_2}-r)x_1=(r-1)^2(y_2-1)x_1+(r-1)(y_2-1)x_1\equiv (r-1)(y_2-1)x_1\pmod{p^a}.
\end{align*}
Thus (\ref{eq:condition4}) can be converted into $(y_{2}-1)x_{1}\equiv y_{1}x_{2}+w\pmod{p^{a-d}}$.
Since by (\ref{eq:lambda1-1}),
\begin{align}
y_1x_2&\equiv\frac{ty_1}{m}(x_1+w-y_2)=w(x_1+w-y_2)\pmod{p^{a+\deg(y_1)-b}} \nonumber \\
&\equiv w(x_1+w-y_2)\pmod{p^{a-d}},   \label{eq:x2y1}
\end{align}
we are led to $(y_2-1)x_1\equiv w(x_1+w-y_2+1)\pmod{p^{a-d}}$, i.e.,
\begin{align}
(y_2-1-w)(x_1+w)\equiv 0\pmod{p^{a-d}};
\end{align}
due to (\ref{eq:w}), this is equivalent to (\ref{eq:lambda1-2}).
\end{proof}

\bigskip

Set
\begin{align}
f(y_{1})=\begin{cases}
2^{a_{2}-d_{2}-1}, &\text{if}\ c_{2}\neq b_{2}, \  \min\{b_{2},c_{2}\}=a_{2}-d_{2}  \ \text{and\ } \deg_{2}(y_{1})=b_{2}-d_{2}, \\
0, & \text{otherwise}.
\end{cases}   \label{eq:f}
\end{align}

\begin{lem} \label{lem:deducing2}
If $2\in\Lambda_{1}$, then the conditions {\rm(\ref{eq:condition2})}$_{2}$ and {\rm(\ref{eq:condition3})}$_{2}$ hold if and only if
\begin{enumerate}
  \item [\rm(i)] if $b_2=c_{2}=d_{2}=1$ (so that $a_{2}=2$), then no additional condition is required;
  \item [\rm(ii)] if $d_{2}=1$ and $\max\{b_2,c_{2}\}>1$, then $2\mid y_{1}$, $\deg_{2}(x_{2})\ge a_{2}-b_{2}-e+1$ and
      \begin{align}
      w\equiv 2^{e-1}(y_{1}-y_{2}+1)\pmod{2^{a_2-1}};   \label{eq:lambda2-0}
      \end{align}
  \item [\rm(iii)] if $d_{2}>1$, then
      \begin{align}
      mx_{2}&\equiv t(x_{1}+w-y_{2})\pmod{2^{a_{2}}}, \label{eq:lambda2-1} \\
      y_{2}&\equiv 1+w+f(y_{1}) \pmod{2^{a_{2}-d_{2}}}.   \label{eq:lambda2-2}
      \end{align}
\end{enumerate}
\end{lem}

\begin{proof}
Abbreviate $a_2,b_2,c_2,d_2, \deg_2(x)$ to $a,b,c,d,\deg(x)$, respectively.

(i) For any $x,u>0$, we have $r^x\equiv 1+2x\pmod{4}$, and
$$[u]_{r^{x}}=\sum\limits_{i=0}^{u-1}r^{ix}\equiv\sum\limits_{i=0}^{u-1}(1+2ix)\equiv u+u(u-1)x\pmod{4}.$$
In particular,
$[m]_{r^{y_2}}\equiv 2+2y_2$, $[t]_{r^{y_1}}\equiv 2+2y_1$, $[r]_{r^{y_1}}\equiv 3+2y_1\pmod{4}.$
The condition (\ref{eq:condition2})$_{2}$, (\ref{eq:condition3})$_2$ can be converted into, respectively,
\begin{align}
(x_2+1)(y_2+1)-(x_1+1)(y_1+1)&\equiv 0\pmod{2}, \label{eq:trivial1} \\
x_2y_1+x_1(1+y_1-y_2)+y_1&\equiv 0\pmod{2}. \label{eq:trivial2}
\end{align}
Due to (\ref{eq:condition00}), $x_2y_1\equiv x_1y_2+1\pmod{2}$, hence (\ref{eq:trivial2}) is equivalent to
$(x_1+1)(y_1+1)\equiv 0\pmod{2}$, which is true since by (\ref{eq:condition00}), at least one of $x_1,y_1$ is odd. Then similarly, (\ref{eq:trivial1}) also holds. %without more condition.

(ii) We first show $2\mid y_{1}$.
Assume on the contrary that $2\nmid y_{1}$. By (\ref{ineq:m}), $b=1$, so that $c>1$. By (\ref{eq:condition2})$_{2}$, $x_2[m]_{r^{y_2}}\equiv 0\pmod{4}$, which forces $2\nmid y_2$: if $2\mid y_2$, then $r^{y_2}\equiv 1\pmod{4}$, so that $4\nmid[m]_{r^{y_2}}$, we would get $2\mid x_2$, contradicting (\ref{eq:condition00}). Then $r^{y_j}\equiv -1\pmod{4}, j=1,2$, and $[r]_{r^{y_1}}\equiv 1\pmod{4}$, so (\ref{eq:condition3})$_2$ implies $2(x_1-x_2)\equiv 0\pmod{4}$. But this contradicts (\ref{eq:condition00}).

Thus $2\mid y_{1}$. By (\ref{eq:condition00}), $2\nmid x_{1}y_{2}$; by (\ref{eq:w}), $2\mid w$.
Hence
\begin{align}
t(x_{1}+w-y_{2})\equiv 0\pmod{2^{a}}. \label{eq:t-xyw}
\end{align}
By (\ref{ineq:b2}), (\ref{ineq:m}), $1+\deg(y_{1})+e\geq b+e\ge a$,
hence
\begin{align}
\deg(r^{y_{1}}-1)=\deg((r^{2})^{y_{1}/2}-1)=e+\deg(y_{1})\ge a-1.
\end{align}
When $c>1$, applying Lemma \ref{lem:simplify}, we obtain
\begin{align*}
%r^{y_{1}}&=(r^{2})^{\frac{y_{1}}{2}}\equiv 1+(r^{2}-1)\cdot\frac{y_{1}}{2}\equiv 1+(r+1)y_{1}, \\
[t]_{r^{y_{1}}}&\equiv(1+2^{e+\deg(y_{1})-1})t\pmod{2^{e+\deg(y_1)+c}} \\
& \equiv t\pmod{2^a},  \\
[r]_{r^{y_{1}}}&=(r^{y_{1}})^{r-1}+[r-1]_{r^{y_{1}}}\equiv 1+(1+2^{e+\deg(y_{1})-1})(r-1)\pmod{2^{e+\deg(y_1)+1}} \\
&\equiv r+2^ey_{1}\pmod{2^a};
\end{align*}
when $c=1$ so that $a=2$, these congruence relations obviously hold.

Due to (\ref{eq:t-xyw}), the condition (\ref{eq:condition2})$_{2}$ becomes $x_{2}[m]_{r^{y_{2}}}\equiv 0\pmod{2^{a}}$.
Since $\deg(r^{y_2}+1)=\deg(r+1)=e$ and $[m]_{r^{y_2}}=(r^{y_2}+1)[m/2]_{r^{2y_2}}$, we have $\deg([m]_{r^{y_2}})=e+b-1$.
Hence
\begin{align*}
\deg(x_{2})\ge a-b-e+1.
\end{align*}
This together with (\ref{ineq:m}) implies
$$\deg((r^{y_{1}}-1)x_{2})=\deg(y_{1})+e+\deg(x_{2})\ge b-1+e+\deg(x_2)\ge a.$$
Then (\ref{eq:condition3})$_{2}$ becomes
\begin{align}
x_1(r^{y_{2}}-r-2^ey_1)\equiv (r-1)w\pmod{2^{a}}. \label{eq:condition5}
\end{align}
Since $\deg(r^{y_{2}-1}-1)=\deg((r^2)^{(y_2-1)/2}-1)=e+\deg(y_2-1)$, we have $r^{y_2-1}-1=2^e(y_2-1)z$ for some odd $z$.
Using $2^{e+1}y_{1}\equiv 2(r-1)w\equiv 0\pmod{2^a}$, we can convert (\ref{eq:condition5}) into (\ref{eq:lambda2-0}).

(iii) Applying Lemma \ref{lem:simplify} (with (\ref{ineq:deg-y1}) recalled), we obtain
\begin{align*}
r^{y_{1}}&\equiv\begin{cases} 1+(r-1)y_{1},  &2\nmid y_{1}, \\ 1+(r-1+2^{2d-1})y_{1}, &2\mid y_{1} \end{cases}\pmod{2^a},  \\
[r]_{r^{y_{1}}}&\equiv r+2^{2d-1}y_{1}\pmod{2^a},  \\
[t]_{r^{y_{1}}}&\equiv (1+2^{d-1}y_{1})t\pmod{2^a},  \\
[m]_{r^{y_{2}}}&\equiv (1+2^{d-1}y_{2})m\pmod{2^a}.
\end{align*}
We deal with the cases $2\mid y_{1}$ and $2\nmid y_{1}$ separately.

(iii-1) If $2\mid y_{1}$, then by (\ref{eq:condition00}), $2\nmid x_{1}y_{2}$, and by (\ref{eq:w}), $2\mid w$.
The condition (\ref{eq:condition2})$_{2}$ becomes
\begin{align}
(1+2^{d-1}y_{2})mx_{2}\equiv t(x_{1}+w-y_{2})\pmod{2^{a}}, \label{eq:condition6}
\end{align}
which can be converted into (\ref{eq:lambda2-1}) via multiplying $1-2^{d-1}y_{2}$. %using (\ref{ineq:alpha1}) and $2\mid x_{1}+w-y_{2}$.
Moreover, (\ref{eq:condition6}) implies $b+\deg(x_{2})\geq\min\{c+1,a\},$ hence
\begin{align*}
2d-1+\deg(x_{2})+\deg(y_{1})&\ge 2d-1+(\min\{c+1,a\}-b)+(b-d) \\
&=d-1+\min\{c+1,a\}\ge a.
\end{align*}
As a result, $x_2(r^{y_1}-1)\equiv (r-1)x_2y_1\pmod{2^a}$.
Using this and $2^{2d-1}(x_1-1)y_1\equiv 0\pmod{2^a}$, we may convert (\ref{eq:condition3})$_{2}$ into
\begin{align}
(r-1)x_{2}y_{1}+2^{2d-1}y_{1}+(r-r^{y_{2}})x_{1}+(r-1)w\equiv 0\pmod{2^{a}}.  \label{eq:condition7}
\end{align}
By an argument similar as that deducing (\ref{eq:deg-y2}) in the proof of Lemma \ref{lem:deducing},
we obtain $\deg(y_{2}-1)\geq a-2d$, and then by Lemma \ref{lem:simplify} (II),
$$r^{y_{2}-1}-1\equiv (1+2^{d-1})(r-1)(y_{2}-1)\pmod{2^{a}}.$$
Using $(r-1)(r^{y_2-1}-1)\equiv 0\pmod{2^a}$, we can further convert (\ref{eq:condition7}) into
\begin{align}
(y_{2}-1)x_{1}\equiv y_{1}x_{2}+w+2^{d-1}(y_{1}-y_{2}+1) \pmod{2^{a-d}}.    \label{eq:condition8}
\end{align}
Similarly as (\ref{eq:x2y1}), it follows from (\ref{eq:lambda2-1}) that
$y_1x_2\equiv w(x_1+w-y_2)\pmod{2^{a-d}}$, and then (\ref{eq:condition8}) becomes
\begin{align}
(y_2-1-w)(x_1+w+2^{d-1})\equiv 2^{d-1}(y_1-w) \pmod{2^{a-d}}.  \label{eq:condition9}
\end{align}
From (\ref{ineq:deg-y1}) and (\ref{ineq:deg-w}) we see that $\deg(y_1-w)\ge a-2d$, and the equality holds if and only if one of the following cases occurs:
\begin{itemize}
  \item $\deg(w)>\deg(y_1)=a-2d$, which is equivalent to $\deg(y_1)=b-d$ and $c>b=a-d$;
  \item $\deg(y_1)>\deg(w)=a-2d$, which is equivalent to $\deg(y_1)=b-d$ and $b>c=a-d$.
\end{itemize}
Thus (\ref{eq:condition9}) becomes
$$(y_2-1-w)(x_1+w+2^{d-1})\equiv f(y_1)\equiv f(y_1)(x_1+w+2^{d-1}) \pmod{2^{a-d}},$$
which is equivalent to (\ref{eq:lambda2-2}).

(iii-2) If $2\nmid y_{1}$, then $d, c\geq b$, and $2d\geq a$.
By Lemma \ref{lem:simplify} (II),
$r^{y_2}\equiv 1+(r-1)y_2\pmod{2^{a}}$,
hence (\ref{eq:condition2})$_{2}$, (\ref{eq:condition3})$_{2}$  become, respectively,
\begin{align}
(1+2^{d-1}y_{2})mx_{2}+ty_{2}&\equiv (1+2^{d-1})tx_{1}+tw\pmod{2^{a}},  \label{eq:condition90}  \\
(y_{2}-1)x_{1}&\equiv y_{1}x_2+w+2^{d-1}x_{1}\pmod{2^{a-d}}.  \label{eq:condition80}
\end{align}

If $c=b$, then by (\ref{eq:w}), $2\mid x_{1}$, and by (\ref{eq:condition00}), $2\nmid x_{2}$. By (\ref{eq:condition90}), $2\mid y_{2}$, and then (\ref{eq:condition90}) becomes (\ref{eq:lambda2-1}). We can reduce (\ref{eq:condition80}) to $y_{2}-1\equiv w\pmod{2^{a-d}}$
similarly as in the proof of Lemma \ref{lem:deducing}.

Now assume $c>b$ so that $2\mid w$. By (\ref{eq:w}), $2\nmid x_1$. Since $c+d-1\geq b+d\geq a$, we can reduce (\ref{eq:condition90}) to (\ref{eq:lambda2-1}) via multiplying by $1-2^{d-1}y_{2}$.
If $2d>a$, then still similarly as in the proof of Lemma \ref{lem:deducing}, we can reduce (\ref{eq:condition80}) to $y_{2}-1\equiv w\pmod{2^{a-d}}$; if $2d=a$, then $b=a-d=d$, then similarly as in (iii-1), we can reduce (\ref{eq:condition80}) to
$y_{2}-1\equiv w+2^{a-d-1}\pmod{2^{a-d}}$.

Thus in any case, (\ref{eq:condition2})$_2$, (\ref{eq:condition3})$_2$ are equivalent to (\ref{eq:lambda2-1}), (\ref{eq:lambda2-2}).
\end{proof}

\subsection{Main result}

Let $m_{0}$ be the smallest positive integer $k$ such that $r^{k}\equiv 1\pmod{p^{a_{p}}}$ for all $p\in\Lambda_{2}$.
Combining Lemma \ref{lem:injective1}, Lemma \ref{lem:injective2}, Remark \ref{rmk:p-II-I1}, Lemma \ref{lem:deducing} and Lemma \ref{lem:deducing2}, we establish
\begin{thm} \label{thm:main}
Each automorphism of $H(n,m;t,r)$ is given by
\begin{align*}
\alpha^{u}\beta^{v}\mapsto \exp_{\alpha}(x_{1}[u]_{r^{y_{1}}}+r^{y_{1}u}x_{2}[v]_{r^{y_{2}}})\beta^{y_{1}u+y_{2}v}, \qquad u,v>0,
\end{align*}
for a unique quadruple $(x_{1},x_{2},y_{1},y_{2})$ with $0<x_{1},x_{2}\le n$, $0<y_{1},y_{2}\le m$ and such that
\begin{enumerate}
  \item [\rm(i)] for all $p\in\Lambda$,
        $$\left\{\begin{array}{lll}
        p \nmid y_{2},  &p\in\Lambda', \\
        p \nmid x_{1}+ty_1/m,  &p\in\Lambda_{2}\ \text{or\ } p\in\Lambda_{1}\ \text{with\ } b_{p}c_p=0,\\
        p \nmid x_{1}y_{2}-x_{2}y_{1},  &p\in\Lambda_{1}\ \text{with\ } b_{p},c_p>0;
        \end{array}\right.$$
  \item [\rm (ii)] $(r-1,t)y_{1}\equiv 0\pmod{m}$ and $y_{1}\equiv y_{2}-1\equiv 0\pmod{m_0}$;
  \item [\rm (iii)] for all $p\in\Lambda_{1}$ with $p\ne 2$  or $p=2, a_{2}=d_{2}$,
        \begin{align*}
        mx_{2}&\equiv t(x_{1}+ty_{1}/m-y_{2})\pmod{p^{a_{p}}}, \\
        y_{2}&\equiv 1+ty_{1}/m\pmod{p^{a_{p}-d_{p}}};
        \end{align*}
  %\item [\rm (iv)] if $b_2=c_{2}=d_{2}=1$ and $a_{2}=2$, then
  %      $2\nmid y_{1}$ and $2\nmid x_1-x_2$;
  \item [\rm (iv)] if $\max\{b_2,c_{2}\}>d_{2}=1$ and $a_2>1$, then
        $2\mid y_{1}$, $\deg_{2}(x_{2})\ge a_{2}-b_{2}-e+1$ and
        \begin{align*}
        ty_{1}/m\equiv 2^{e-1}(y_{1}-y_{2}+1)\pmod{2^{a_{2}-1}};
        \end{align*}
  \item [\rm(v)] if $d_{2}>1$, then
        \begin{align*}
        mx_{2}&\equiv t(x_{1}+ty_{1}/m-y_{2})\pmod{2^{a_{2}}}, \\
        y_{2}&\equiv 1+ty_{1}/m+f(y_{1}) \pmod{2^{a_{2}-d_{2}}}.
        \end{align*}
\end{enumerate}
\end{thm}

\end{document}